\documentclass[11pt]{35}
\usepackage{amsthm, amsmath, amssymb, amsbsy, amsfonts, latexsym, euscript}
\usepackage{calrsfs} 
\usepackage{graphicx}

\DeclareMathAlphabet{\varmathbb}{U}{pxsyb}{m}{n}

\def\leq{\leqslant}

\def\geq{\geqslant}
\def\phi{\varphi}

\def\kappa{\varkappa}

\newcommand{\D}{\mathrm{d}\kern0.2pt}%

\newcommand{\E}{\mathrm{e}\kern0.2pt} 
\newcommand{\ii}{\kern0.05em\mathrm{i}\kern0.05em}

\newcommand{\RR}{\mathbb{R}}%

\newtheorem{theorem}{\bf \indent Theorem}[section]
\newtheorem{proposition}{\bf \indent Proposition}[section]

\newtheorem{corollary}{\bf \indent Corollary}[section]

\theoremstyle{remark}

\newtheorem{definition}{\bf\indent Definition}[section]

\numberwithin{equation}{section}

\begin{document}

\noindent {\Large \bf Mean value properties of solutions to the Helmholtz \\[3pt]
and modified Helmholtz equations}

\vskip5mm

{\bf N. Kuznetsov}

\vskip-2pt {\small Laboratory for Mathematical Modelling of Wave Phenomena}
\vskip-4pt {\small Institute for Problems in Mechanical Engineering} \vskip-4pt
{\small Russian Academy of Sciences} \vskip-4pt {\small V.O., Bol'shoy pr. 61, St.
Petersburg, 199178} \vskip-4pt {\small Russian Federation} \vskip-4pt {\small
nikolay.g.kuznetsov@gmail.com}

\vskip7mm

\parbox{146mm} {\noindent Mean value properties of solutions to the $m$-dimensional
Helmholtz and modified Helm\-holtz equations are considered. An elementary
derivation of these properties is given; it involves the Euler--Poisson--Darboux
equation. Despite the similar form of these properties for both equations, their
consequences distinguish essentially. The restricted mean value property for
harmonic functions is amended so that a function, satisfying it in a bounded domain
of a special class, solves the modified Helmholtz equation in this domain.}

\vskip10pt

{\centering
\section{Introduction}
}

\noindent In this note, we consider mean value properties of solutions to the $m$-dimensional
Helmholtz and modified Helmholtz equations. These equations have the next level of
complexity after the Laplace equation $\nabla^2 u = 0$ ($\nabla = (\partial_1, \dots
, \partial_m)$ is the gradient operator and $\partial_i = \partial / \partial x_i$).
Indeed, the Helmholtz equation
\begin{equation}
\nabla^2 u + \lambda^2 u = 0 , \quad \lambda \in \RR \setminus \{0\} , \label{Hh}
\end{equation}
and the modified Helmholtz equation
\begin{equation}
\nabla^2 u - \mu^2 u = 0 , \quad \mu \in \RR \setminus \{0\} , \label{MHh}
\end{equation}
have the simplest possible additional terms which differ only by the sign of the
coefficient at $u$. Solutions of these equations are assumed to be real; indeed, the
results obtained for them can be extended to complex-valued ones by considering the
real and imaginary part separately.

Since \eqref{Hh} turns into \eqref{MHh} by putting $\lambda = \ii \mu$, it is
natural to expect that mean value properties of solutions to these equations have
similar form and this is really so. However, these properties are separated in some
sense by those of harmonic functions; moreover, when we turn to their consequences
such as various versions of maximum principle and Liouville's theorem the results
differ essentially for solutions of \eqref{Hh} and \eqref{MHh}. Like harmonic
functions solutions of these equations are real analytic functions, but  nonzero
constants, which are harmonic, are not their solutions.

A few words about early studies of mean value properties of solutions to equations
\eqref{Hh} and \eqref{MHh}. It was Weber, who first derived the mean value formula
for spheres for solutions to \eqref{Hh} in three dimensions; see his paper \cite{W1}
published in 1868. In two dimensions, he obtained the analogous theorem next year;
see \cite{W2}. However, the $m$-dimensional case was considered only in the
classical book \cite{CH} (pp.~288 and 289) on the basis of the Pizzetti's mean value
formula for polyharmonic functions \cite{Piz}. To the best author's knowledge, no
further studies of mean value properties for solutions to \eqref{Hh} had appeared
since publication of the first edition of \cite{CH} in 1937 until 2021, when the
mean value formula for balls and converse theorems for balls and spheres were
obtained in \cite{Ku1}.

Even less was published about mean value properties of solutions to equation
\eqref{MHh}. The mean value formula for spheres was derived by C.~Neumann in the
three-dimensional case; see his book \cite{NC}, Chapter~9, \S~3, published in 1896.
The $m$-dimensional formula for spheres appeared without proof in the Poritsky's
article \cite{Po} published in 1938, but its results had been obtained ten years
earlier being presented to the AMS in December 1928. Mistakingly, C.~Neumann's
formula is attributed to Weber in \cite{Po}. The mean value formula for balls and
converse theorems for balls and spheres were formulated in \cite{Ku1} by analogy to
those for solutions of the Helmholtz equation.

Since mean value properties of solutions to the Helmholtz equation were studied in
detail in the recent paper \cite{Ku1}, our attention is focussed primarily on the
modified Helmholtz equation in the present note.

\vskip10pt 

{\centering \section{New derivation of the mean value formulae} }

By $B_r (x) = \{ y : |y-x| < r \}$ we denote the open ball of radius $r$ centred at
$x \in \RR^m$; it is called admissible with respect to a domain $D \subset \RR^m$
provided $\overline{B_r (x)} \subset D$, whereas $\partial B_r (x)$ is called
admissible sphere in this case. Let $f$ be a function continuous on $D$, then
\[  M^\circ (f, x, r) = \frac{1}{|\partial B_r|} \int_{\partial B_r (x)} \!\!
f (y) \, \D S_y \ \ \mbox{and} \ \ M^\bullet (f, x, r) = \frac{1}{|B_r|} \int_{B_r
(x)} \!\! f (y) \, \D y
\]
are its mean values over admissible sphere and ball, respectively. Here $\D S$ is
the surface area measure, $|\partial B_r| = m \, \omega_m r^{m-1}$ is the area of
$\partial B_r$ and $|B_r| = \omega_m r^m$ is the volume of $B_r$; the volume of unit
ball is $\omega_m = 2 \, \pi^{m/2} / [m \Gamma (m/2)]$.

It is clear that these functions are continuous in $x$ and $r$; moreover, if $u \in
C^k (D)$, then its mean values are in the same class in $x$ and $r$. By continuity
we have that
\[ M^\bullet (f, x, 0) = M^\circ (f, x, 0) = f (x) ,
\]
whereas further identities for $M^\circ$ can be found in \cite{J}, Chapter IV.

\vskip7pt {\bf 2.1. The Euler--Poisson--Darboux equation.} Let us show that the
obvious relation
\begin{equation}
m \int_0^r t^{m-1} M^\circ (u, x, t) \, \D t = r^m  M^\bullet (u, x, r) \, ,
\label{M+M}
\end{equation}
where $u$ is a $C^2$-function, yields that
\begin{equation}
M^\circ_{rr} + (m-1) \, r^{-1} M^\circ_r = \nabla^2_x M^\circ \ \ \mbox{for} \ r > 0
\, , \label{EPD}
\end{equation}
which is referred to as the Euler--Poisson--Darboux equation.

Indeed, applying the Laplacian to both sides of \eqref{M+M}, we obtain
\[ m \, \omega_m \int_0^r t^{m-1} \nabla^2_x M^\circ (u, x, t) \, \D t = 
\int_{B_r (0)} \!\! \nabla^2_x u (x + y) \, \D y \, .
\]
By Green's first formula the last integral is equal to
\[ \int_{|y| = r} \!\! \nabla_x \, u (x+y) \cdot \frac{y}{r} \, \D S_y \, ,
\]
and changing variables this can be written as follows:
\[ r^{m-1} \frac{\partial}{\partial r} \int_{|y|=1} \!\! u (x + r y) \, 
\D S_y = m \, \omega_m r^{m-1} M_r^\circ (u, x, r) \, .
\]
Thus we arrive at the equality
\begin{equation}
r^{m-1} M_r^\circ (u, x, r) = \int_0^r t^{m-1} \nabla^2_x M^\circ (u, x, t) \, \D t
\label{M_r}
\end{equation}
Differentiation of this relation with respect to $r$ yields \eqref{EPD}.

\vskip7pt {\bf 2.2. Mean value formulae for spheres.} To distinguish a solution of
equation \eqref{Hh} from that of \eqref{MHh}, we denote the latter by $v$ in what
follows, keeping $u$ for a solution of \eqref{Hh}. Thus, from now on we write
equation \eqref{MHh} as follows:
\begin{equation}
\nabla^2 v - \mu^2 v = 0 , \quad \mu \in \RR \setminus \{0\} . \label{MHh+}
\end{equation}
As usual, $J_\nu$ and $I_\nu$ denote the Bessel function and the modified Bessel
function of order $\nu$, respectively. Now, we are in a position to prove the
following.

\begin{theorem}
Let $D$  be a domain in $\RR^m$. If $u$ and $v \in C^2 (D)$ solve equations
\eqref{Hh} and \eqref{MHh+}, respectively, then the mean value equalities for
spheres
\begin{equation}
M^\circ (u, x, r) =  a^\circ (\lambda r) \, u (x) \ \ \mbox{and} \ \ M^\circ (v, x,
r) = \widetilde a^\circ (\mu r) \, v (x) \label{MM}
\end{equation}
hold for every admissible ball $B_r (x)$. Here
\begin{equation}
a^\circ (\lambda r) = \Gamma \left( \frac{m}{2} \right) \frac{J_{(m-2)/2} (\lambda
r)}{(\lambda r / 2)^{(m-2)/2}} \ \ \mbox{and} \ \ \widetilde a^\circ (\mu r) =
\Gamma \left( \frac{m}{2} \right) \frac{I_{(m-2)/2} (\mu r)}{(\mu r / 2)^{(m-2)/2}}
\, ,
\label{a}
\end{equation}
respectively.
\end{theorem}

\begin{proof}
We begin with a solution $v$ of equation \eqref{MHh+} and then turn to the
corresponding result for $u$ solving \eqref{Hh}.

It is straightforward to show that $a (r) = \widetilde a^\circ (\mu r)$ is a unique
solution the following Cauchy problem:
\begin{equation}
a_{rr} + (m-1) \, r^{-1} a_r - \mu^2 a = 0 , \ \ a (0) = 1 , \ \ a_r (0) = 0 \, .
\label{g5}
\end{equation}
This follows by virtue of the relations
\begin{equation}
z I_{\nu+1} (z) + 2 \nu I_\nu (z) - z I_{\nu-1} (z) = 0 \, , \ \ \ [z^{-\nu} I_\nu
(z)]' = z^{-\nu} I_{\nu+1} (z) \, ; \label{diff}
\end{equation}
see \cite{Wa}, p. 79. In particular, the second one implies the second initial
condition.

For the function $w (r, x) = \widetilde a^\circ (\mu r) \, v (x) - M^\circ (v, x,
r)$, which is defined for all $x \in D$ and all $r \geq 0$ such that $\partial B_r
(x)$ is admissible, we have
\[ w (x, 0) = 0 \, , \quad w_r (x, 0) = 0 \, .
\]
The first initial condition is a consequence of the equalities $\widetilde a^\circ
(0) = 1$ and $M^\circ (v, x, 0) = v (x)$, whereas the second one follows in the
limit as $r \to 0$ from equation \eqref{EPD} multiplied by $r$. Moreover, equations
\eqref{EPD} and \eqref{g5} yield that
\[ w_{rr} + (m-1) \, r^{-1} w_r - \mu^2 w = 0 \ \ \mbox{for} \ r > 0 \, .
\]
Since the latter Cauchy problem has only a trivial solution, we obtain that the
second equality \eqref{MM} holds with the coefficient given by the second formula
\eqref{a}.

To obtain the first equality \eqref{MM} one has to repeat these considerations with
$a (r) = a^\circ (\lambda r)$ and $w (r, x) = a^\circ (\lambda r) \, u (x) - M^\circ
(u, x, r)$; moreover, the relations 
\begin{equation}
z J_{\nu+1} (z) - 2 \nu J_\nu (z) + z J_{\nu-1} (z) = 0 \, , \ \ \ [z^{-\nu} J_\nu
(z)]' = - z^{-\nu} J_{\nu+1} (z) \label{diff'}
\end{equation}
must be used instead of \eqref{diff}; see \cite{Wa}, pp. 45 and 66, where these
formulae are derived.
\end{proof}

For $m=3$ the mean value formulae for spheres have particularly simple form. Indeed,
\[ a^\circ (\lambda r) = \frac{\sin \lambda r}{\lambda r} \ \ \mbox{and} \ \ \widetilde
a^\circ (\mu r) = \frac{\sinh \mu r}{\mu r} \, ,
\]
because $J_{1/2} (z) = \sqrt{2 / (\pi z)} \sin z$ and $I_{1/2} (z) = \sqrt{2 / (\pi
z)} \sinh z$. For arbitrary $m$, we see that $a^\circ (\lambda r)$ and $\widetilde
a^\circ (\mu r) \to 1$ as $\lambda, \mu \to 0$, and so both equalities \eqref{MM}
recover in the limit the mean value theorem for spheres valid for harmonic
functions. However, the second formula \eqref{diff} shows that $\widetilde a^\circ
(\mu r)$ approaches unity from above, whereas $a^\circ (\lambda r)$ does this from
below according to the analogous differentiation formula for $J_\nu$.

\vskip7pt {\bf 2.3. Mean value formulae for balls.} An immediate consequence of
\eqref{MM} and \eqref{a} are mean value properties for balls.

\begin{corollary}
Let $D$  be a domain in $\RR^m$. If $u$ and $v \in C^2 (D)$ solve equations
\eqref{Hh} and \eqref{MHh+}, respectively, then the mean value equalities for balls
\begin{equation}
M^\bullet (u, x, r) =  a^\bullet (\lambda r) \, u (x) \ \ \mbox{and} \ \ M^\bullet
(v, x, r) = \widetilde a^\bullet (\mu r) \, v (x) \label{MM'}
\end{equation}
hold for every admissible ball $B_r (x)$. Here
\begin{equation}
a^\bullet (\lambda r) = \Gamma \left( \frac{m}{2} + 1 \right) \frac{J_{m/2} (\lambda
r)}{(\lambda r / 2)^{m/2}} \ \ \mbox{and} \ \ \widetilde a^\bullet (\mu r) = \Gamma
\left( \frac{m}{2} + 1 \right) \frac{I_{m/2} (\mu r)}{(\mu r / 2)^{m/2}}  \, ,
\label{a'}
\end{equation}
respectively.
\end{corollary}

\begin{proof}
To obtain the first formula \eqref{MM'} with the coefficient given in \eqref{a'} it
suffices to write the first formula \eqref{MM} in the form
\[ m \, \omega_m a^\circ (\lambda \rho) \, u (x) = \int_{\partial B_1 (0)} \!\! u 
(x + \rho y) \, \D S_y \, ,
\]
multiply by $\rho^{m-1}$, and integrate with respect to $\rho$ within an admissible
ball. To this end formula 1.8.1.21, \cite{PBM}:
\begin{equation}
\int_0^x \!\! x^{1 + \nu} J_\nu (x) \, \D x = x^{1 + \nu} J_{\nu + 1} (x) \, , \ \
\Re \, \nu > -1 . \label{PBM}
\end{equation}
is helpful. Applying it with $\nu = (m-2)/2$ while integrating the first expression
\eqref{a}, the required result follows.

For obtaining the second formula \eqref{MM'} with the coefficient given in
\eqref{a'} one has to use the same procedure, but applying formula \eqref{PBM} with
Bessel functions changed to modified ones; see 1.11.1.5 in \cite{PBM}.
\end{proof}

As in the case of spheres, $a^\bullet (\lambda r)$ and $\widetilde a^\bullet (\mu
r)$ tend to unity as $\lambda, \mu \to 0$ (from below and above, respectively), thus
recovering in the limit the mean value theorem for balls valid for harmonic
functions.

\vskip10pt

{\centering \section{Applications of mean value properties} 
}

First we notice an interesting feature of the function $a^\circ (\lambda r)$ for
$m=2$. In this case, the sequence $\{ J_0 ( j_{0,n} |x| / r) \}_{n=1}^\infty$ (as
usual, $j_{\nu,n}$ denotes the $n$th positive zero of $J_\nu$) consists of
eigenfunctions satisfying equation \eqref{Hh} in $B_r (0)$ and the Dirichlet
boundary condition on $\partial B_r (0)$ provided $\lambda^2$ is determined from the
relation $\lambda r = j_{0,n}$. Since $J_0 (0) = 1$, the validity of the mean value
formula for $\partial B_r (0)$ for each of these eigenfunctions is guaranteed by the
equality $a^\circ (j_{0,n}) = 0$. On the other hand, all other Dirichlet
eigenfunctions vanish at the origin, and so the mean value formula holds for them
irrespective of the value attained by $a^\circ$.

Similarly, the sequence $\{ J_0 ( j_{1,n} |x| / r) \}_{n=1}^\infty$ consists of
eigenfunctions satisfying equation \eqref{Hh} in $B_r (0)$ and the Neumann boundary
condition on $\partial B_r (0)$ provided $\lambda^2$ is determined from the relation
$\lambda r = j_{1,n}$; notice that $J'_0 (z) = - J_1 (z)$. Each of these
eigenfunctions has the zero mean value for $B_r (0)$, being orthogonal to a nonzero
constant---the Neumann eigenfunction corresponding to $\lambda^2 = 0$. Despite the
equality $J_0 (0) = 1$, the mean value formula for $B_r (0)$ holds for each of these
eigenfunctions; this time because $a^\bullet (j_{1,n}) = 0$. As above, all Neumann
eigenfunctions without central symmetry vanish at the origin, and so the mean value
formula for $B_r (0)$ holds for them irrespective of the value attained by
$a^\bullet$.

In the remaining part of this section, the results obtained in \S~2 are used for
proving several properties of solutions to equations \eqref{Hh} and \eqref{MHh+}.

The second formula \eqref{diff'} implies that $a^\circ$ decreases monotonically on
the interval $(0, j_{m/2,1})$; also, it is positive on the smaller interval $(0,
j_{(m-2)/2,1})$ and changes sign at $\lambda r = j_{(m-2)/2,1}$ (it is clear that
the function $a^\bullet$ has similar properties). Combining the latter fact and the
first equality \eqref{MM}, we arrive at the following.

\vspace{2.6mm}

\begin{proposition}
Let $D$ be a domain in $\RR^m$, and let $u \in C^2 (D)$ do not vanish identically
and satisfy equation \eqref{Hh} in $D$ with $\lambda > 0$. If for some $x \in D$
there exists $B_{r_*} (x) \subset D$ with $r_* > j_{(m-2)/2,1} / \lambda$, then $u$
vanishes somewhere in $B_{r_*} (x)$. Moreover, the nodal set \[ N (u) = \{ y \in D :
u (y) = 0 \} \] is an $m-1$-dimensional hypersurface in $D$.
\end{proposition}

\begin{proof}
If $u (x) = 0$, then a zero already exists and we denote it by $x_0$. Let $u (x) >
0$, and let $\delta > 0$ be sufficiently small and such that $r_+ = (j_{(m-2)/2,1} +
\delta) / \lambda < r_*$. Then $M^\circ (u, x, r_+) < 0$ because $a^\circ (\lambda
r_+) < 0$ in view of the behaviour of $J_{(m-2)/2} (\lambda r)$. Therefore, $u$
attains a negative value at some point on $\partial B_{r_+} (x) \subset B_{r_*}
(x)$. By continuity $u$ vanishes at some point $x_0 \in B_{r_*} (x)$ lying between
$x$ and the obtained point on $\partial B_{r_*} (x)$.

If $u (x) < 0$, then applying the same considerations to $-u$ we obtain the required
point $x_0$ such that $-u (x_0) = 0$.

Thus in all possible cases, there exists $x_0 \in B_{r_*} (x)$ such that $u (x_0) =
0$, and so $M^\bullet (u, x_0, r)$ vanishes provided $\overline {B_r (x_0)} \subset
D$. This implies that $B_r (x_0)$ is divided by an analytic $m-1$-dimensional
hypersurface into parts, where $u > 0$ and $u < 0$; indeed, $u$ is a real analytic
function in $D$ not vanishing identically there. By analyticity this hypersurface
extends from $B_r (x_0)$ to $D$.
\end{proof}

To illustrate Proposition~3.1, let us consider the function $u_+ (x) = |x|^{-1} \sin
\lambda |x|$. It satisfies equation \eqref{Hh} in $\RR^3$, and $N (u_+) =
\cup_{k=1}^\infty S_k$---the union of spheres $S_k = \partial B_{\pi k / \lambda}
(0)$---is its nodal set. Being positive in $B_{\pi / \lambda} (0)$, this function
demonstrates that the restriction $r_* > j_{(m-2)/2,1} / \lambda$ ($j_{(m-2)/2,1} /
\lambda = \pi / \lambda$ in this case) is essential. Moreover, $u_+$ easily provides
various domains enclosing balls of the radius $r_* > \pi / \lambda$, and so
containing parts of $S_k$ or the whole ones as its nodal sets in these domains. For
example, the domain $B_{(\pi + \epsilon) / \lambda} (0)$, where $\epsilon > 0$ is
sufficiently small, has the whole $S_1$ as the closed $N (u_+)$. On the other hand,
$N (u_+)$ consists of two pieces within $B_{(\pi + \epsilon) / \lambda} ((3 \pi / 2,
0, 0))$ ($\epsilon$ is the same as above): one lying on $S_1$ and the other one on
$S_2$. Of course, there is a plethora of more complicated examples.

We recall that mean value properties of harmonic functions have two important
consequences: the strong maximum principle and Liouville's theorem. The first
asserts that a function harmonic in a domain $D$ cannot have local maxima or minima
in $D$; moreover, if it is continuous in $\overline D$, which is bounded, then its
maximum and minimum are attained on $\partial D$; see \cite{M}, Chapter~11, \S~8.
The second theorem says that every harmonic function bounded below (or above) on
$\RR^m$ is a constant; see \cite{M}, Chapter~12, \S~4. Let us consider whether mean
value properties of solutions to equations \eqref{Hh} and \eqref{MHh+} imply
analogous theorems.

It is clear that the function $u_+$ violates both these assertions, and so they do
not hold for solutions of equation \eqref{Hh}. How this is related to the
inequalities $a^\circ (\lambda r) < 1$ and $a^\bullet (\lambda r) < 1$ for $\lambda
r > 0$ is an open question.

Another question is about analogous theorems for solutions of equation \eqref{MHh+}.
They are also not true as formulated above; indeed, the function $u_- (x) = |x|^{-1}
\sinh |x|$ satisfies this equation on $\RR^3$ (with $\mu = 1$), and has the local
(and global) minimum equal to one at the origin. Moreover, $u_-$ violates the
generalized Liouville theorem (see \cite{ABR}, p.~198), which guarantees that a
harmonic function is constant on $\RR^m$ under a weaker assumption than in
Liouville's theorem.

Thus some extra restrictions must be imposed in order to convert formulations of the
maximum principle and Liouville's theorem into true ones for solutions of equation
\eqref{MHh+}. It is easy to obtain them for the whole $\RR^m$, in which case, in
view of self-similarity, it is sufficient to restrict ourselves to the equation:
\begin{equation}
\nabla^2 v - v = 0 \ \ \mbox{in} \ \RR^m ; \label{MHh1}
\end{equation}
indeed, it follows from \eqref{MHh+} by the proper change of variables.

Let us consider the asymptotic behaviour of the function $\widetilde a^\circ (r)$ as
$r \to \infty$. The formula
\begin{equation}
I_\nu (z) = \frac{\E^z}{\sqrt{2 \pi z}} \left[ 1 + O (|z|^{-1}) \right] \, , \ \
|\arg z| < \pi /2 , \label{asym_I}
\end{equation}
whose principal term does not depend on $\nu$, is valid as $|z| \to \infty$; see
\cite{Wa}, p.~80. Combining \eqref{asym_I} and the second formula \eqref{a}, one
obtains
\[ \widetilde a^\circ (r) = \frac{\Gamma (m / 2) \, 2^{(m-3)/2}}{\sqrt \pi}
\frac{\E^r}{r^{(m-1)/2}} \left[ 1 + O (r^{-1}) \right] \ \ \mbox{as} \ r \to \infty \, .
\]
Now we are in a position to prove the following version of Liouville's theorem.

\begin{theorem}
Let $v$ be a solution of \eqref{MHh1} on $\RR^m$. If the inequality
\begin{equation}
| v (x) | \leq C (1 + |x|)^n \ \ holds \ for \ all \ x \in \RR^m \label{N}
\end{equation}
with some $C > 0$ and a nonnegative integer $n$, then $v$ vanishes identically.
\end{theorem}

\begin{proof}
Let us write the second formula \eqref{MM} in the form
\[ m \, \omega_m \widetilde a^\circ (r) \, v (x) = \int_{\partial B_1 (0)} \!\! v
(x + r y) \, \D S_y \, .
\]
Then \eqref{N} and the asymptotic formula for $\widetilde a^\circ (r)$ imply the
inequality
\[ | v (x) | \leq \widetilde C (1 + |x| + r)^n \frac{r^{(m-1)/2}}{\E^r} \ \ 
\mbox{for all} \ x \in \RR^m \ \mbox{and all} \ r > 0
\]
with some $\widetilde C > 0$. Letting $r \to \infty$, the required assertion
follows.
\end{proof}

Inequality \eqref{N} implies that a solution to equation \eqref{MHh1} is trivial
regardless that $n > 0$ can be taken arbitrarily large. On the other hand, if the
same inequality is imposed on a harmonic function, then it is a (harmonic)
polynomial, whose degree is less than or equal to $n$; see \cite{V}, p.~290.

The next assertion concerns the behaviour of $|v|$ for a nontrivial solution $v$ to
equation \eqref{MHh+}.

\begin{proposition}
Let $D$ be a domain in $\RR^m$, and let a nonvanishing identically function $v \in
C^2 (D)$ satisfy equation \eqref{MHh+} there. Then for every $x \in D$ there exists
$y \in D$ such that $|v (y)| > |v (x)|$.
\end{proposition}

\begin{proof}
Without loss of generality, we assume that $v (x) \geq 0$; indeed, $-v$ ought to be
considered instead of $v$ otherwise. Then we have that $M^\circ (v, x, r) \geq 0$
for an admissible sphere $\partial B_r (x)$ because $\widetilde a^\circ (\mu r) >
1$; moreover, $v (x) < M^\circ (v, x, r)$ by Theorem~2.1. The last inequality
implies that there exists $y \in \partial B_r (x) \subset D$ such that $v (y) > v
(x)$, which completes the proof.
\end{proof}

An immediate consequence of this proposition is the weak maximum principle for
solutions to equation \eqref{MHh+}; see \cite{GT}, \S~3.1, for the approach applicable to
general elliptic equations.

\begin{theorem}
Let $D$ be a bounded domain in $\RR^m$. If $v \in C^2 (D) \cap C^0 (\overline D)$
satisfies equation \eqref{MHh+} in $D$, then
\begin{equation}
\sup_{x \in D} | v (x) | = \max_{x \in \partial D} | v (x) | \, . \label{wmp}
\end{equation}
\end{theorem}

\begin{proof}
In the case of nonvanishing identically function $v$, we take a sequence
$\{x_k\}_{k=1}^\infty \subset D$ such that $|v (x_k)| \to \sup_{x \in D} | v (x) |$
as $\ k \to \infty$. Since $D$ is bounded, $\{x_k\}_{k=1}^\infty$ has a limit point
in $\overline D$, say $x_0$, and $|v (x_0)| = \sup_{x \in D} | v (x) |$ by
continuity. Moreover, $x_0 \in \partial D$; indeed, if $x_0 \in D$, then
Proposition~3.2 implies that there exists $y \in D$ such that $|v (y)| > |v (x_0)|$,
which is impossible. Now \eqref{wmp} follows from the equality $\sup_{x \in D} | v
(x) | = \max_{x \in \overline D} | v (x) |$ valid for $v \in C^0 (\overline D)$.
\end{proof}

An immediate consequence of this theorem (see \cite{NC}, p.~260, for the original
formulation) is the uniqueness of a solution to the Dirichlet problem for equation
\eqref{MHh+} in a bounded domain as well as the continuous dependence of solutions
to this problem on boundary data.

\vspace{-3mm}

{\centering \section{A converse of Theorem 2.1 concerning equation (2.4)} }

In the recent paper \cite{Ku1}, it is shown that if the first (second) equality
\eqref{MM} is valid for every $x \in D$ and for all $r \in (0, r (x))$, where $B_{r
(x)} (x)$ is admissible, then $u$ ($v$, respectively) is a solution of the Helmholtz
(modified Helmholtz, respectively) equation. Thus, mean value properties
characterize solutions of these equations in the same way as the Koebe theorem
characterizes harmonic functions; see the classical Kellogg's book \cite{K},
pp.~226, 227. The next step was made by Kellogg; his paper \cite{K1} of 1934
initiated a long series of publications dealing with the so-called restricted mean
value properties in order to characterize harmonicity. Baxter generalized Kellogg's
result in his note \cite{B1} and introduced the following definition of the
restricted properties.


\begin{definition}
A real-valued function $f$ defined on an open $G \subset \RR^m$ is said to have the
restricted mean value property with respect to balls (spheres) if for each $x \in G$
there exists a single ball (sphere) centred at $x$ of radius $r (x)$ such that $B_{r
(x)} (x) \subset G$ and the average of $f$ over this ball (its boundary) is equal
to $f (x)$.
\end{definition}

Some applications of these properties and their further generalizations used in the
theory of harmonic functions were outlined in the survey paper \cite{Ku2}, \S~3. In
the present note, Definition~4.1 is amended in order to accommodate it for
characterizing solutions of the modified Helmholtz equation.


\begin{definition}
A real-valued function $f$ defined on an open $G \subset \RR^m$ is said to have the
modified restricted mean value property with respect to spheres if for each $x \in
G$ there exists a single sphere centred at $x$ of radius $r (x)$ such that $B_{r
(x)} (x) \subset G$ and the second equality \eqref{MM} with $r = r (x)$ holds for
$f$.
\end{definition}

\begin{theorem}
Let $D \subset \RR^m$ be a bounded domain such that the Dirichlet problem for the
modified Helmholtz equation is soluble in $C^2 (D) \cap C^0 (\overline D)$ for
every continuous function given on $\partial D$. If $v \in C^0 (\overline D)$ has
the modified restricted mean value property in $D$ with respect to spheres, then $v$
satisfies equation \eqref{MHh+} in $D$.
\end{theorem}

\begin{proof}
First, let us show that the theorem's assumptions yield that
\begin{equation}
\max_{x \in \overline D} |v (x)| = \max_{x \in \partial D} |v (x)| \, . \label{max}
\end{equation}
Reasoning by analogy with the proof of Proposition~3.2, we see that the modified
restricted mean value property implies that for every $x \in D$ there exists $y \in
D$ such that $|v (y)| > |v (x)|$. Then the considerations used in the proof of
Theorem~3.2 yield \eqref{max}.

Let $f$ denote the trace of $v$ on $\partial D$; then there exists $v_0 \in C^0
(\overline D)$ solving the Dirichlet problem for equation \eqref{MHh+} in $D$ with
$f$ as the boundary data. Hence $v_0$ satisfies the second equality \eqref{MM} for
all $x \in D$ and all admissible $\partial B_r (x)$, and so the modified restricted
mean value property is valid for $v - v_0$. Then the weak maximum principle
\eqref{max} holds for $v - v_0$, thus implying that $v \equiv v_0$ in $D$ because $v
\equiv v_0$ on $\partial D$. Then $v$ also satisfies the second equality \eqref{MM}
for all $x \in D$ and all admissible $\partial B_r (x)$. Now, Theorem~7 formulated
in \cite{Ku1} yields that $v$ is a solution of equation \eqref{MHh+} in $D$.
\end{proof}

Since the proof of the latter theorem is omitted in \cite{Ku1}, we include it in
this note for the sake of completeness.

\begin{theorem}
Let $D$ be a bounded domain in $\RR^m$. If a real-valued $v \in C^0 (D)$ satisfies
the second equality \eqref{MM} with some $\mu > 0$ for every $x \in D$ and all $r
\in (0, r (x))$, where the ball $B_{r (x)} (x)$ is admissible, then $v$ is a
solution of equation \eqref{MHh+} in $D$.
\end{theorem}

\begin{proof}
First, let is necessary to show that $v$ is smooth and for this purpose we use the
method applied by Mikhlin in his proof of the corresponding theorem for harmonic
functions; see \cite{M}, Chapter~11, \S~7.

Let $D'$ be a subdomain of $D$ whose closure $\overline{D'} \subset D$ is separated
from $\partial D$ by a layer formed by parts of balls located within $D$; each of
these balls has its centre on $\partial D$ and radius $2 \epsilon$. By
$\omega_\epsilon (|y - x|) = \omega_\epsilon (r)$ we denote the mollifier considered
in \cite{M}, Chapter~1, \S~1. Let $x \in D'$, then, multiplying the second equality
\eqref{MM} by $\omega_\epsilon (r)$, we obtain
\[ v (x) \, \widetilde a^\circ (\mu r) \, |\partial B_r| \, \omega_\epsilon (r) = 
\omega_\epsilon (r) \int_{\partial B_r (x)} \!\! v (y) \, \D S_y \, ,
\]
where $\widetilde a^\circ$ is defined by \eqref{a}. Integration with respect to $r$
over $(0, \epsilon)$ yields
\[ v (x) \, c (\mu, \epsilon) = \int_{B_\epsilon (x)} \!\! v (y) \, 
\omega_\epsilon (|y - x|) \, \D y = \int_{D} \!\! v (y) \, \omega_\epsilon (|y - x|)
\, \D y \, .
\]
Here the last equality follows from the fact that $x \in D'$, whereas
$\omega_\epsilon (|y - x|)$ vanishes outside $B_\epsilon (x)$. Moreover,
\[ c (\mu, \epsilon) = \int_0^\epsilon \widetilde a^\circ (\mu r) \, |\partial
B_r| \, \omega_\epsilon (r) \, \D r > 0 \, ,
\]
because $\widetilde a^\circ (\mu r) > 1$. Since $\omega_\epsilon$ is infinitely
differentiable, the obtained representation shows that $v \in C^\infty (D')$. By
taking $\epsilon$ arbitrarily small, we see that $v \in C^\infty (D)$.

Now we are in a position to show that $v$ is a solution of equation \eqref{MHh+} in
$D$. Let $x \in D$ and let $r (x) > 0$ be such that $B_{r (x)} (x)$ is admissible.
Since the second equality \eqref{MM} holds for all $r \in (0, r (x))$, for any such
$r$ the second equality \eqref{MM'} holds as well (it follows from \eqref{MM} by
integration). Applying the Laplacian to the integral on the left-hand side of the
second equality \eqref{MM'}, we obtain
\[ \int_{|y| < r} \!\! \nabla^2_x \, v (x+y) \, \D y = \int_{|y| = r} \!\! \nabla_x \,
v (x+y) \cdot \frac{y}{r} \, \D S_y \, .
\]
Here the equality is a consequence of Green's first formula. By changing variables
this can be written as follows:
\[ r^{m-1} \frac{\partial}{\partial r} \int_{|y|=1} \!\! v (x + r y) \, 
\D S_y = |\partial B_1 (0)| r^{m-1} \frac{\partial}{\partial r} M^\circ (v, x, r) \, ,
\]
where $M^\circ (v, x, r) = \widetilde a^\circ (\mu r) \, v (x)$. In view of the
second equality \eqref{diff}, we have that
\[ \frac{\partial}{\partial r} M^\circ (v, x, r) = - \frac{\mu I_{m/2} (\mu r)}
{(\mu r / 2)^{(m-2)/2}} \, v (x) \, .
\]
Combining the above considerations and the second equality \eqref{MM'}, we conclude
that for every $x \in D$ the equality
\[ \int_{|y| < r} \!\! [ \nabla^2_x \, v - \mu^2 v ] \, (x+y) \, \D y = 0 \ \ 
\mbox{holds for all} \ r \in (0, r (x)) \, .
\]
Thus, in each $B_r (x)$ there exists $y (r, x)$ such that $[ \nabla^2 \, v - \mu^2 v
] \, (y (r, x)) = 0$. Since $y (r, x) \to x$ as $r \to 0$, we obtain by continuity
that $v$ satisfies the modified Helmholtz equation at every $x \in D$.
\end{proof}

The question about domains in which the Dirichlet problem for an elliptic equation
is soluble has a long history which goes back to George Green's \textit{Essay on the
Application of Mathematical Analysis to the Theories of Electricity and Magnetism}
published in 1828. This problem for the Laplace equation was posed in it for the
first time. The final answer when the Dirichlet problem for harmonic functions has a
solution was given by Wiener \cite{NW} in 1924; the notion of capacity was
introduced for this purpose.

The class of domains such that the Dirichlet problem is soluble is the same for the
modified Helmholtz equation and for the Laplace equation. This follows from the
results of Oleinik \cite{O} and Tautz \cite{T}; they demonstrated independently and
published in 1949 that this fact about the solubility of the Dirichlet problem is a
common characteristic which is true for elliptic equations of rather general form
(see the monograph \cite{Mi}, Chapter~IV, \S~28, for a review of related papers).

\vspace{-12mm}

\renewcommand{\refname}{
\begin{center}{\Large\bf References}
\end{center}}
\makeatletter
\renewcommand{\@biblabel}[1]{#1.\hfill}
\makeatother

\end{document}